 \documentclass[11pt,a4paper,final]{amsart}

\usepackage[cp850]{inputenc}

\usepackage{caption}
\usepackage{graphicx, tabularx, color}
\usepackage{amssymb}
\usepackage{amsmath}
\usepackage{ifpdf}
\usepackage[all]{xy}

\usepackage{latexsym}

\newtheorem{theorem}{Theorem}[section]
\newtheorem{proposition}[theorem]{Proposition}
\newtheorem{remark}[theorem]{Remark}
\newtheorem{convention}[theorem]{Convention}
\newtheorem{lemma}[theorem]{Lemma}
\newtheorem{corolario}[theorem]{Corollary}

\newcommand{\funmix}{\mathcal{F}}
\newcommand{\fa}{f_{\mathbb{R}}}

\newcommand{\va}{\varphi_{\mathbb{R}}}

\newcommand{\R}{\mathbb{R}}

\newcommand{\Demo}{{\it Proof: }}

\newcommand{\conf}[1]{\mathrm{Conf}(#1)}

\newcommand{\scal}[1]{R(#1)}
\newcommand{\fv}[1]{dv(#1)}
\newcommand{\fvr}[1]{d\sigma(#1)}

\newcommand{\tricc}[2]{R_{#2}(#1)}
\newcommand{\grad}[1]{grad(#1)}

\newcommand{\ghy}{F}
\newcommand{\Cab}{\mathcal{C}_{a,b}}

\newcommand{\mean}[1]{H(#1)}
\newcommand{\opens}{\Omega}

\newcommand{\bifspace}{{\mathcal{M}}}
\newcommand{\II}{\mathrm I\!\mathrm I}

\newcommand{\biffun}[1]{{\mathcal{F}^{#1}}}

\newcommand{\biffunm}[1]{\mathcal{F}^{#1}}
\newcommand{\bifspacem}{\mathcal{M}}

\newcommand{\bord}[1]{\Sigma_{#1}}
\newcommand{\second}[2]{\mathrm I\!\mathrm I_{#2}(#1)}
\newcommand{\cone}{\mathbb{M}^k(M)}

\begin{document}

\title[Rigidity for Yamabe-type problems]{Local rigidity for Yamabe-type problems in warped products}

\author[S. C. Garc\'ia-Mart\'inez]{S. Carolina Garc\'ia-Mart\'inez}
\address{Departamento de Matem\'aticas\hfill\break\indent  Universidad Nacional de Colombia\hfill\break\indent  Bogot\'a, Colombia.}
\email{sacgarciama@unal.edu.co}
\author[J. Herrera]{Jonatan Herrera}
\address{Instituto de Matem\'atica e Estat\'istica\hfill\break\indent Universidade de S\~ao Paulo\hfill\break\indent Rua do Mat\~ao, 1010 - Cidade Universitaria\hfill\break\indent  S\~ao Paulo, Brazil}
\email{jonatanhf@gmail.com}

\begin{abstract}
In this paper we study the local rigidity of metrics defined on a compact manifold $M$ with boundary $\partial M$ and satisfying both constant scalar curvature on $M$ and constant mean curvature in $\partial M$. We present some geometrical hypotheses ensuring local rigidity first for the general Riemannian case, and then, for the warped metric one. These conditions arise from the study of a spectral problem which is not included within the classical problems (Neumann, Steklov,\dots) and that we call ``mixed eigenvalue problem''. Finally, we apply our results for the spatial slice of the Anti-de Sitter spacetimes.\\\\\\
\textbf{2010 Mathematics Subject Classification Numbers.} Primary 53A10, Secundary 58C40, 83C57.\\\\
\textbf{Key Words and Phrases:} local rigidity, Yamabe problem with boundary, Anti-de Sitter Schwarzschild spacetimes, metric variations, warped product, mixed eigenvalue problem.
\end{abstract}
\maketitle
\section{Introduction}
The classical Yamabe problem consists in showing that every Riemannian compact manifold, without boundary, admits a conformally related metric with constant scalar curvature. In dimension two, such a problem follows from the Uniformization Theorem of Riemann Surfaces. For higher dimensions, the problem was formulated by Yamabe in \cite{Yamabe}, where the first steps towards a proof of the existence of a solution were given. After the combined efforts of Trudinger \cite{Trudinger}, Aubin \cite{Aubin} and Schoen \cite{Schoen}, the Yamabe problem was completely solved. It is important to highlight that constant scalar curvature metrics can be characterized variationally as critical points of the Hilbert-Einstein functional in conformal classes. It is known that the minimum of
this functional in a conformal class is unique, see \cite{Anderson}. However, in many cases a rich variety of constant scalar curvature metrics arise as critical points that are not necessarily minimizers, and for this reason, it is very interesting to find conformal classes where the Yamabe problem has multiple solutions.
Among others, bifurcation techniques can be applied in order to obtain multiplicity results.
The existence of multiple solutions for the Yamabe problem in different settings has long been studied in the literature, see for instance \cite{AM,BP, BP2, LPZ2}.

Now, if the compact manifold considered in the Yamabe Problem has a nonempty boundary, several possible boundary conditions can be studied. For instance, from the point of view of conformal geometry, a geometrical condition would have to involve the mean curvature of the boundary. This important observation led Escobar in 1992 (see \cite{E2}) to study the problem of finding smooth metrics with constant scalar curvature and minimal boundary inside a given conformal class. In a natural continuation of the latter, the same author in \cite{E3} analyzed the conformal deformation of a metric to a scalar flat metric with constant mean curvature on the boundary generalizing the famous Riemann Mapping Theorem to higher dimensions. Subsequently, he also obtained results for the \textit{Yamabe problem with boundary} under mixed constraints (see \cite{E}):
\begin{align}\label{c} a \,V + b \,A= 1
\end{align}
where $V$ and $A$ are the volume of manifold and the area of its boundary, respectively, and with $a\geq0$ and $b$ real numbers. More precisely, he studied a problem of existence of conformal metrics with constant scalar curvature $\scal{g}$ in the manifold and constant mean curvature $\mean{g}$ on the boundary, which are related by \[(n-1)b\scal{g}= 2n\mean{g}a.\] The solutions of this problem are critical points of the functional formulated by Gibbons-Hawking and York \cite{GH,York} (which we will refer to as the \emph{GHY-functional}) restricted to the constraint \eqref{c}. The existence of such critical points was proved by Escobar for manifolds of nonpositive type (i.e., the first eigenvalue of the conformal Laplacian with zero boundary condition is nonpositive) and for almost any manifold of positive type if $b$ is sufficiently small.  In \cite{HL1} and \cite{HL2}, Han and Li also showed the existence of solutions when $a > 0$, the manifold is of positive type and either is locally conformally flat with umbilic boundary, or $n\geq5$ and the boundary has a non-umbilic point.

\smallskip

Taking advantage of the variational approach, our aim in this paper is to study the local rigidity of families of metrics defined on a fixed compact manifold $M$ with boundary $\partial M$ and satisfying both constant scalar curvature in $M$ and constant mean curvature on $\partial M$. Roughly speaking, a family $\mathcal F$ of metrics is (locally) rigid if given another metric $g$ solution of the Yamabe problem and sufficiently close to some element of $\mathcal F$, then $g$ belongs to $\mathcal F$.

For our results, we will give conditions to ensure the non-singularity of the second variation of the GHY-functional under the constraint \eqref{c}, which will ensure the rigidity as a consequence of the Implicit Function Theorem (see \cite[Appendix A]{LPZ2}). Such a conditions will be obtained by making lower estimates of the first eigenvalue of a particular kind of spectral problem, which we have called {\em mixed eigenvalue problem} (see Section \ref{mixed} for a general background on this problem).

When we consider warped product spaces with constant scalar curvature, a special family of metrics appears. As it is well known, warped spaces present a natural foliation given by hypersurfaces with constant mean curvature (its fibres). Under the assumption that certain initial fibre of the space is minimal, the foliation is naturally identified with an one-parametric family of metrics satisfying previous conditions, and so, susceptible for our results on rigidity. Such results can be interpreted as local uniqueness of the foliation from a metric viewpoint. 

As an application of the above, we prove that three-dimensional warped products with monotonically increasing warping function are always globally rigid (see Theorem \ref{three-rigid}). This result is specially interesting when the fibre of the warped product is a sphere since it becomes the spatial slice of the Anti-de Sitter Schwarzschild spacetime, a classical relativistic model of a spacetime containing a black hole. Our result allows us to deduce the local uniqueness of the model under metric variations preserving both the causal structure (i.e., its conformal structure) and the boundary of the black hole (see \cite{Wa} for a background on general relativity and the Anti-de Sitter model).

%
\smallskip

This work is organized as follows. In Section \ref{section2}, besides to fixing the notation, we include some preliminaries on bifurcation theory, as well as the criterion that we will use to prove local rigidity. In Section \ref{mixed}, we discuss a general background about the spectrum of the mixed eigenvalue problem, including classical results as the Courant's nodal Theorem and the Rayleigh characterization for eigenvalues (which will be necessary in the sequel). Moreover, we present some results for the particular case of warped spaces.

In Section \ref{section3}, we first introduce the variational setting for the Yamabe problem with boundary, as we described before. Then, we obtain our main results about the stability of the solutions of the Yamabe problem for the general Riemannian case (Theorem \ref{tege}). The case of warped spaces is studied in Section \ref{wm} where first we introduce the one-parametric family of metrics associated to the foliation \eqref{defmetric}, and then we obtain the main result for the warped case (Corollary \ref{rigidwarpedmetric2}). Additionally, in Section \ref{ADS}, we present several examples where our results are applicable, including the case of the spatial fibre of the Anti-de Sitter Schwarzschild model. We have also included in Section \ref{sec5.2} a general procedure to study both the metric rigidity and bifurcation in warped spaces not covered by our general results. Finally, in the Appendix (Section \ref{appendix}), 
we compute the first and second variation of the GHY-functional. This computation has been performed by several authors (see \cite{A} for instance), but we include it here for the sake of completeness.

\section{Preliminaries}\label{section2}

The purpose of this section is to state the basic elements, results and notation that we are going to use in the rest of the paper. Let $(M^n,g)$ be an arbitrary $n$-dimensional Riemannian compact manifold with non-empty boundary $\partial M$ and $n\geq 3$. We will assume that the boundary $\partial M$ is an $(n-1)$-dimensional smooth manifold (i.e., a hypersurface) formed by different connected components. Additionally, we will also assume that such a components are grouped in two disjoint sets $\partial M=\Sigma_1\cup \Sigma_2$ (see for instance Figure \ref{boundary}). As we will see in subsequent sections, our intention here is to group all the components sharing the same (constant) mean curvature, allowing us to work with them as a single component.

\begin{figure}[htbp]
\centering
\ifpdf
  \setlength{\unitlength}{1bp}%
  \begin{picture}(370.90, 179.42)(0,0)
  \put(0,0){\includegraphics{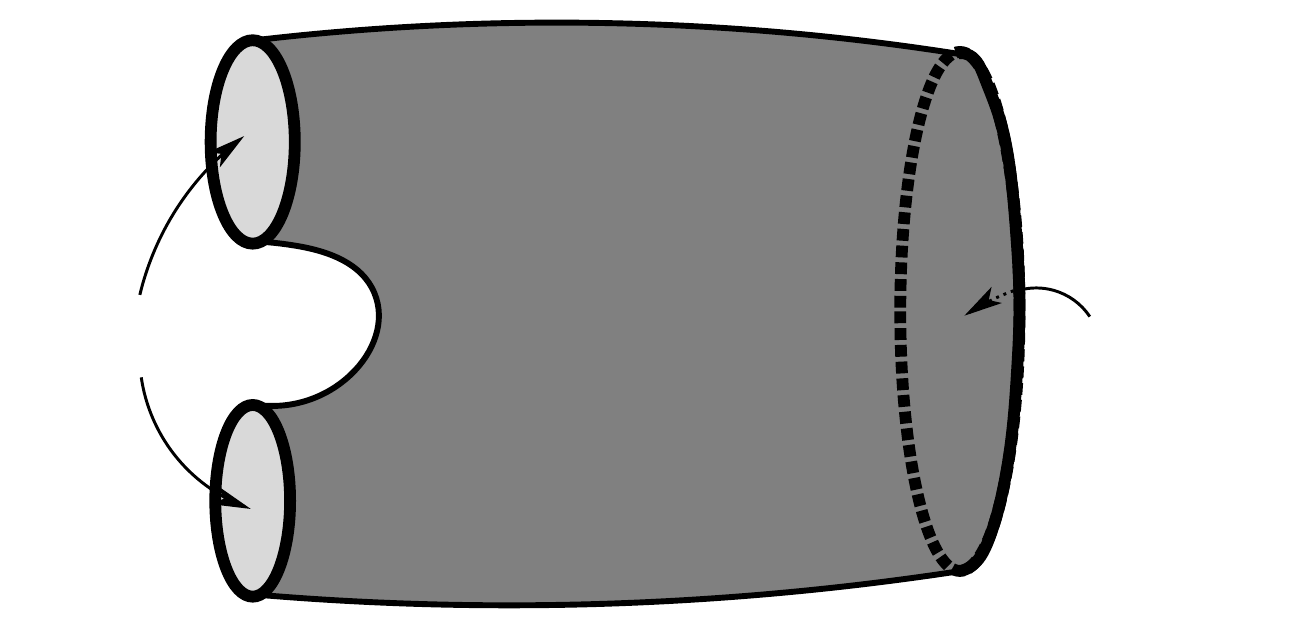}}
  \put(154.72,85.33){\fontsize{24.23}{27.07}\selectfont $M$}
  \put(25.67,78.89){\fontsize{16.54}{18.24}\selectfont $\Sigma_1$}
  \put(315.23,80.84){\fontsize{16.54}{18.24}\selectfont $\Sigma_2$}
  \end{picture}%
\else
  \setlength{\unitlength}{1bp}%
  \begin{picture}(370.90, 179.42)(0,0)
  \put(0,0){\includegraphics{dibujo1}}
  \put(154.72,85.33){\fontsize{14.23}{17.07}\selectfont $M$}
  \put(5.67,82.89){\fontsize{8.54}{10.24}\selectfont $\Sigma_1$}
  \put(315.23,84.84){\fontsize{8.54}{10.24}\selectfont $\Sigma_2$}
  \end{picture}%
\fi
\caption{\label{boundary}Visual interpretation of our setting.}
\end{figure}
For a given system of local coordinates $x_1,\cdots,x_n$ around a point $p$, the metric $g$ and the volume element will be represented as
\[g=g_{ij}\,dx^idx^j,\qquad \fv{g}=\sqrt{\mathrm{det}(g_{ij})}\,dx^1\wedge\cdots\wedge dx^n,\]
with indexes varying from $1\leq i,j\leq n$. When $p$ belongs to $\partial M$, we will assume that the coordinate $x_n$ is normal to the hypersurface $\partial M$ and pointing inward. In particular, the area element restricted to the boundary will take the following form: \[\fvr{g}=\sqrt{\mathrm{det}(g_{\alpha\beta})}\,dx^1\wedge\cdots\wedge dx^{n-1}\]
where $1\leq \alpha,\beta\leq n-1$. \footnote{As a convention, we will always assume that, when we work in coordinates, the indexes $i,j$ will vary between $1,\cdots, n$, while the indexes $\alpha, \beta$ will vary in $1,\cdots, n-1$.} Finally, the scalar curvature in $(M,g)$ will be denoted by $R(g)$, while the mean curvature on each $\Sigma_k$ will be denoted by $H(g_{k})$ for $k=1,2$.

\smallskip

\subsection{Bifurcation theory on metric variations}\label{sec2.1}

Here we introduce some basic framework about bifurcation theory that we will need along the rest of the paper. We refer the reader to \cite{LPZ2,S-W} and the references therein for general background.

As a first step to define the concepts of rigidity and bifurcation, we have to define a natural norm on the space of Riemannian metrics. For this, let us denote by $S^{k}(M)$, with $k\geq 2$, the space of all symmetric $(0,2)$-tensors of class $C^k$ defined on $M$. The space of Riemannian metrics on $M$, denoted here by $\bifspacem$,  is a subspace of $S^{k}(M)$ with the structure of an open cone and tangent space $T_g\bifspacem$ naturally identifiable with the entire $S^{k}(M)$. Now, consider an auxiliary Riemannian metric $g_A$ in $M$ defining both a connection $\nabla_A$ and an inner product $\left\langle\left\langle\cdot,\cdot\right\rangle\right\rangle_{A}$ on the space $S^k(M)$. Then, we define the norm $||\cdot||_{C^{k}}$  given by

%

\[
||T||_{C^{k}}={\rm max}_{j=1,\dots,k}\left[{\rm max}_{p\in M} ||\nabla^{(j)}_{A}T(p)||_{A}\right].
\]

For $g\in \bifspacem$, let us denote by\footnote{Here, $\mathcal{H} ^1(M)$ is the Sobolev space of $L^2$ functions on $M$ with first derivatives in $L^2$.}  

\begin{equation}\label{eq:defconf}
\conf g=\left\{fg: f\in \mathcal{H}^1(M)\quad \textrm{and} \quad f>0 \right\}
\end{equation} 

\noindent the space of $\mathcal{H}^1(M)$-conformal metrics of $g$, which is clearly identified as an open subset of $\mathcal{H}^1(M)$. We will consider in $\conf g$ the differential structure induced by $\mathcal{H}^1(M)$.  


\smallskip

Now, let us consider an one-parameter family of metrics $\{g_{\lambda}\}_{\lambda\in I}\subset \bifspacem$, where $I$ is an arbitrary open interval on $\mathbb{R}$. We will assume that, for all $\lambda$, the metric $g_{\lambda}$ has:
\begin{itemize}
\item[(a)] constant scalar curvature $R(g_{\lambda})$ in $M$,
\item[(b)] constantly zero mean curvature on $\Sigma_1$ (i.e., $\Sigma_1$ is a minimal hypersurface) and
\item[(c)] constant mean curvature (CMC for short) $H(g_2)$ on $\Sigma_2$.
\end{itemize}
In this setting, $\lambda_{*}$ is said to be a {\em point of bifurcation} if there exist a sequence $\{\lambda_n\}_n\subset I$ and a sequence $\{g_n\}_n\subset \bifspacem$ satisfying:

\begin{itemize}
\item[(1)] $\lim\limits_{n\rightarrow \infty}\lambda_n=\lambda_*$ and $\lim\limits_{n\rightarrow \infty}g_n=g_{\lambda_*}$ (the latter with the topology induced by $||\cdot||_{C^{k}}$),
\item[(2)] For all $n$, $g_n$ determines constant scalar curvature in $M$; and makes $\Sigma_1$ a minimal hypersurface and $\Sigma_2$ a CMC-hypersurface. Moreover, the scalar curvature in $M$ and the constant mean curvature on $\Sigma_2$ are $R(g_{\lambda_n})$ and $H((g_{\lambda_n})_2)$, respectively,
\item[(3)] For all $n$, $g_n$ belongs to the $\mathcal{H}^1(M)$-conformal class of $g_{\lambda_n}$, but $g_n \neq g_{\lambda_n}$.
\end{itemize}
If $\lambda_*$ is \emph{not} a bifurcation point, we will say that the family $\{g_{\lambda}\}_{\lambda}$ is {\em locally rigid} at $\lambda_*$.

\smallskip

In the literature, we can find several criteria for the existence of bifurcation points as well as for (local) rigidity. In general, such results require a variational problem whose critical points are the objects of interest. For our particular case, given a family $\{g_{\lambda}\}_{\lambda}$ of metrics, we need a path of $C^k$-functionals (with $k\geq 2$) $\biffun{\lambda}:\conf {g_{\lambda}}\rightarrow \R$ whose critical points are metrics $g$ in the conformal class of $g_\lambda$ with constant scalar curvature $R(g)=R(g_{\lambda})$ for a fixed $\lambda\in I$; and making $\Sigma_1$ minimal and $\Sigma_2$ of constant mean curvature with $H(g_2)=H((g_{\lambda})_2)$. Then, if the metric $g_{\lambda_*}$ is a non-degenerate critical point for some $\lambda_*$, the fiber bundle version of the Implicit Function Theorem ensures the local rigidity on $\lambda_*$ (see \cite[Appendix]{LPZ2} for details). For bifurcation, in addition to the degeneracy of $g_{\lambda_*}$, we also require a variation in the associated Morse index to the critical points in a close neighbourhood of $\lambda_*$. Namely, let us denote by $i(\mathcal{F},g)$ the Morse index of $g$, that is, the dimension of the maximal subspace of the tangent space $T_g\bifspacem$ where the second variation $\delta^2(\mathcal{F})_g$ is negative definite. Then, $\lambda_*$ is a bifurcation point if for any $\lambda_1,\lambda_2$ close enough to $\lambda_*$ with $\lambda_1<\lambda_*<\lambda_2$, $i(\mathcal{F}^{\lambda_1},g_{\lambda_1})\neq i(\mathcal{F}^{\lambda_2},g_{\lambda_2})$ (see \cite[Theorem A.2]{LPZ2} for details).
	
In this work, we will focus our attention on local rigidity, even so our studies allow us to do some analysis involving bifurcation points (see Section \ref{sec5.2}). In particular, the desired non-degeneracy will follow by showing that all the eigenvalues of the associated Jacobi operator (which is diagonalizable under our conditions) are positive. In this case, the spectrum of the Jacobi operator lead us to a {\em mixed} eigenvalue problem.


\section{Spectrum of the mixed eigenvalue problem}\label{mixed}
Our aim in this section is to give a general background of the so-called mixed eigenvalue problem, giving the basic properties that we will need through this paper. This kind of problem cannot be consider a Neumann problem nor a Steklov one but a mix of them, can be described in the following way

\begin{equation}\label{mixequation}
\left\{\begin{array}{rlrr}
\triangle_g f = \left(G + \overline{\beta}\right) f & & \hbox{in $M$,}\\\\
-\dfrac{\partial}{\partial n}f=\left(J + \overline{\beta}\right)f & & \hbox{on $\partial M$}
\end{array}\right.
\end{equation}
where $\triangle_g$ is the Laplace-Beltrami operator with nonnegative spectrum\footnote{i.e., $\Delta_g=-\mathrm{div}_g(\mathrm{grad})$.}; $G:M\rightarrow \R$ and $J:\partial M\rightarrow \R$ are functions; $\partial/\partial n$ denotes the inward normal derivative; and $f\in \mathcal{H}^1(M)$.

\smallskip

To what extent are the classical procedures for the study of, say, Dirichlet problems applicable (or, at least, adaptable) to the mixed one? As a first step to answer this question, it is essential to describe the mixed problems variationally. Let us define two bilinear forms $\mathcal{D}$ and $\mathcal{E}$ in the following way:

\begin{align*}
{\mathcal D}(\varphi,\psi)&=\int_{M}\big[g\left(\grad \varphi,\grad \psi\right)-G\varphi\psi\big] \fv{g}- \int_{\partial M}J\varphi\psi \,\fvr{g}\\
{\mathcal E}(\varphi,\psi)&=\int_{M}\varphi\psi \fv{g} + \int_{\partial M} \varphi\psi \fvr{g}.
\end{align*}
Then, for a fixed $\overline{\beta}\in\R$, we define the functional

\[\funmix(\varphi)=\mathcal{D}(\varphi)-\overline{\beta}\mathcal{E}(\varphi)  \]
where $\mathcal{D}(\varphi):=\mathcal{D}(\varphi,\varphi)$ and $\mathcal{E}(\varphi):=\mathcal{E}(\varphi,\varphi)$;
and observe that the first variation of such a functional becomes:
\[
\delta\funmix_{\varphi}(\psi)=2\left({\mathcal D}(\varphi,\psi)-\overline{\beta} {\mathcal E}(\varphi,\psi)\right).
\]
Now, the first Green identity leads us to:

{\small \begin{align*}
{\mathcal D}(\varphi,\psi)-\overline{\beta} {\mathcal E}(\varphi,\psi)= & \int_{M}\psi\left(\triangle_g\varphi-G\varphi-\overline{\beta}\varphi\right)\fv{g} +\int_{\partial M}\psi\left(-\partial_n \varphi -J\varphi-\overline{\beta}\varphi\right)\fvr{g}.
\end{align*}}

Hence, $\varphi$ is a critical point for the functional $\funmix$, if and only if it is a solution for the mixed eigenvalue problem \eqref{mixequation}. By using this variational approach, we can re-obtain some classical and well-known results for the mixed case. For instance, the Courant's Nodal Theorem follows directly (see \cite[Page 452]{Courant}) as well as the
%
classical characterization for the eigenvalues was obtained by Courant \cite{Courant} and Rayleigh \cite{Rayleigh}. In fact, it follows that the eigenvalues are determined by a sequence $\{\overline{\beta}_i\}_{i\in \mathbb{N}}$, repeated according to their multiplicity, and such that ${\rm lim}_{i\rightarrow\infty}\,\overline{\beta}_i=\infty$. Moreover,

\begin{align*}\label{ray}
\overline{\beta}_n=\min_{\varphi\in \{\varphi_1,\dots,\varphi_{n-1}\}^\bot} \frac{\mathcal{D}(\varphi)}{\mathcal{E}(\varphi)}
\end{align*}

\noindent where each $\varphi_i$ (with $1\leq i\leq n-1$) is the eigenfunction associated to $\overline{\beta}_i$ and

\[
\{\varphi_1,\dots,\varphi_{n-1}\}^\bot=\{\varphi\in \mathcal{H}^1(M):\mathcal{E}(\varphi,\varphi_i) =0,\; \forall i=1,\dots,n-1\}.
\]
For convenience, when $n=1$ the minimum is taken on the whole $\mathcal{H}^1(M)$.
\smallskip

Now, observe that the first Green identity lead us to:  

\[\def\arraystretch{2.2}\begin{array}{r>{\displaystyle}l}
\displaystyle \int_{M}|\grad{\varphi}|_g^2 \,\fv{g}= &\int_{M}\varphi\Delta_g{\varphi}\,\fv{g}- \int_{\partial M} \varphi \partial_n \varphi \fvr{g}\\
= & \int_{M}\left(\overline{\beta}+G\right)\varphi^2 \fv{g} + \int_{\partial M}\left(\overline{\beta}+J\right)\varphi^2 \fvr{g}\\
= & \overline{\beta}\left(\int_{M}\varphi^2 \fv{g} + \int_{\partial M}\varphi^2 \fvr{g}\right) \\
& + \int_{M}G\varphi^2 \fv{g} + \int_{\partial M} J\varphi^2 \fvr{g}.
\end{array}
\]
Then, if we assume the following normalization
\begin{equation}\label{auxeq1}\int_{M}\varphi^2 \fv{g} + \int_{\partial M}\varphi^2 \fvr{g} =
 {\mathcal {E}(\varphi,\varphi)}= 1,
\end{equation}
any eigenvalue $\overline{\beta}$ should satisfy
\begin{equation}\label{auxeq2}
\overline{\beta}= 
\int_{M}|\grad{\varphi}|_g^2 \,\fv{g}-\int_{M}G\varphi^2 \fv{g} - \int_{\partial M}J \varphi^2 \fvr{g}.
\end{equation}

In particular, we can obtain easily the following simple lower estimate for the eigenvalues





\begin{proposition}\label{lowboundg} The eigenvalues $\overline{\beta}$ satisfies
\[\overline{\beta}\geq -\left(G^+ +J^+\right).
\]
where $G^+=\max_M \{0,G\}$ and $J^+=\max_{\partial M}\{0,J\}$. 
\end{proposition}

\Demo  First, let us recall that $G$ (and analogously $J$) can be expressed as $G=G^+-G^-$ where $G^+=\max_M \{0,G\}$ and $G^-=\max_M \{0,-G\}$. Thus, from \eqref{auxeq2} we have that

\[\def\arraystretch{2.2}\begin{array}{rl}
\overline{\beta}=&\displaystyle \int_{M}|\grad{\varphi}|_g^2\, \fv{g}-\displaystyle\int_{M}G\varphi^2 \fv{g} - \displaystyle\int_{\partial M}J \varphi^2 \fvr{g}\geq\\ \geq & -G^+\displaystyle\int_{M}\varphi^2 \fv{g}-J^+\displaystyle\int_{\partial M} \varphi^2 \fvr{g}. \end{array}
\]
Using now the normalization in (\ref{auxeq1}), we have that both previous integrals are less than or equal to $1$, so we get
\[
\overline{\beta}\geq -\left(G^++J^+\right).
\]
\qed

Previous lower bound of the eigenvalues will be clearly insufficient to obtain rigidity results in the forthcoming sections. However, for numerical approaches, it determines a start-point for algorithms looking for the first eigenvalue.

\subsection{The mixed eigenvalue problem for warped metric spaces}\label{mixedwarpedsection}
In this section we will focus our attention on the warped product spaces, that is, an $n$-dimensional Riemannian space $(M,g)$ where
\[
M=(r_1,r_2)\times P, \qquad g=dr^2+\alpha^2(r)g^P,
\]

\noindent $(P,g^P)$ is an $(n-1)$-dimensional closed Riemannian manifold and $\alpha$ is a $\mathcal{C}^k$ positive function with $k\geq 2$ 
. Along this section, we will also assume that the previous functions $G$ and $J$ only depends on the parameter $r$. With these assumptions, the mixed eigenvalue problem can be reduced to a Sturm-Liouville-type problem with boundary conditions. In fact,
 when $g$ is a warped metric, the Laplace-Beltrami operator splits as

\[
\triangle_g \varphi=-\frac{1}{\alpha^{n-1}}\partial_r\left(\alpha^{n-1}\partial_r\varphi\right) + \dfrac{1}{\alpha^{2}}\triangle_P\,\varphi
\]

\noindent where $\triangle_{P}$ denotes the Laplace-Beltrami operator on $(P,g^P)$. Therefore, the first equation in (\ref{mixequation}) becomes

\begin{eqnarray}\label{ep}
-\dfrac{1}{\alpha^{n-1}}\partial_r\left(\alpha^{n-1}\partial_r\varphi\right) + \dfrac{1}{\alpha^{2}}\triangle_P\,\varphi=\left(G+\overline{\beta}\right)\varphi.
\end{eqnarray}
Next, by separation of variables, we assume that a solution $\varphi:M\rightarrow \R$ of the above problem can be split as,
\[\varphi(r,x)=\varphi_{\R}(r)\varphi_P(x)\] where $\va:(r_1,r_2)\rightarrow \R$ and $\varphi_P:P\rightarrow \R$. Moreover, if we assume that $\varphi_P\equiv \varphi_P^i$ is a non-zero eigenfunction for $\triangle_{P}$ associated to an eigenvalue $\beta_i$ then \eqref{ep} can be written as follows
\[
-\frac{1}{\alpha^{n-1}}\partial_r\left(\alpha^{n-1}\partial_r\va\right)+\frac{1}{\alpha^{2}}\beta_i\va= \left(G + \overline{\beta}\right)\va
\]

\noindent which is a Sturm-Liouville equation on the interval $(r_1,r_2)$. Hence, joining previous equation with the initial conditions on (\ref{mixequation}) we obtain, for each eigenvalue $\beta_i$ of $\triangle_{P}$, the following problem\footnote{By notation, the dot will denote derivative of a real function.}:

\begin{equation}
\begin{array}{rl}
-\dfrac{1}{\alpha^{n-1}}\partial_r\left(\alpha^{n-1}\partial_r\va\right)+\dfrac{1}{\alpha^{2}}\beta_i\va=& \left(G + \overline{\beta}\right)\va,\\\\
-\dot{\va}(r_1)=&\left(J(r_1)+\overline{\beta}\right)\va(r_1),\\\\
\dot{\va}(r_2)=&\left(J(r_2)+\overline{\beta}\right)\va(r_2).
\end{array}
\end{equation}
It is well-known that for each $i$, these problems admit a sequence of eigenvalues $\{\overline{\beta}^i_j\}_{j\in \mathbb{N}}$ with

\[
\overline{\beta}_1^i<\overline{\beta}_2^i\leq \cdots \leq \overline{\beta}_j^i\leq\cdots
\]

\noindent Therefore, if we denote by $(\va)^i_j$ the eigenfunction associated to $\overline{\beta}_j^i$, we obtain that $\varphi^i_j=(\va)^{i}_j(\varphi_P)_i$ is an eigenfunction for the mixed eigenvalue problem (\ref{mixequation}) with eigenvalue $\overline{\beta}^i_j$. Moreover, from previous expression of the eigenfunctions and the dimension of $\mathcal{H}^1((r_1,r_2)\times P)$, we deduce that all the eigenfunctions are obtained by this process, i.e., we have proved
\begin{theorem}\label{teomixedproblemw}
Let $\varphi$ be an eigenfunction of \eqref{mixequation} with associated eigenvalue $\overline{\beta}$. Then $\varphi$ can be written as the product of two functions $\va:(r_1,r_2)\rightarrow \R$ and $\varphi_P:P\rightarrow \R$ where $\varphi_P$ is an eigenfunction of $\triangle_{P}$ with associated eigenvalue $\beta$, and $\va$ is a solution of the following problem:
\begin{equation}\label{mixproblemwarped}
\begin{array}{rl}
-\dfrac{1}{\alpha^{n-1}}\partial_r\left(\alpha^{n-1}\partial_r\va\right)+\dfrac{1}{\alpha^{2}}\beta\va=& \left(G + \overline{\beta}\right)\va,\\\\
-\dot{\va}(r_1)=&\left(J(r_1)+\overline{\beta}\right)\va(r_1),\\\\
\dot{\va}(r_2)=&\left(J(r_2)+\overline{\beta}\right)\va(r_2).
\end{array}
\end{equation}

\end{theorem}

\smallskip

The remainder of the section is devoted to obtain some additional information about the eigenvalues of this mixed problem on the warped case. As before, let us denote by $\{\beta_i\}$ the ordered sequence of eigenvalues of the Laplacian $\triangle_{P}$ and by $\overline{\beta}_1^i$ the first eigenvalue of the Sturm-Liouville problem \eqref{mixproblemwarped} with $\beta=\beta_i$. Then, the Rayleigh quotient applied to Problem \eqref{mixproblemwarped} reads:

\begin{equation}\label{rquotientwarped}
\overline{\beta}_{1}^{i}=\min_{\va\in \mathcal{H}^1(r_1,r_2)} \frac{-\alpha^{n-1}\va\dot{\va}\Big|_{r_1}^{r_2} + \displaystyle\int_{r_1}^{r_2}\alpha^{n-1}\left[(\dot{\va})^2+(\alpha^{-2}\beta_i-G)\va^2\right]dr}{\displaystyle\int_{r_1}^{r_2}\alpha^{n-1}\va^2\,dr}.
\end{equation}
With this characterization, we are able to prove the following two results:

\begin{lemma}\label{ineeigenvalues}
For all $i\in\mathbb{N}$,
\[
\overline{\beta}_{1}^{i}\leq\overline{\beta}_{1}^{i+1}.
\]
\end{lemma}

\Demo For a fixed function $\va$, the function on the right of (\ref{rquotientwarped}) is non-decreasing respect to $\beta_i$. Therefore,

\[\def\arraystretch{2.2}
\begin{array}{r>{\displaystyle}l}
\overline{\beta}_1^i\!\!&=  \min_{\va\in \mathcal{H}^1(r_1,r_2)} \frac{-\alpha^{n-1}\va\dot{\va}\Big|_{r_1}^{r_2} + \displaystyle\int_{r_1}^{r_2}\alpha^{n-1}\left[(\dot{\va})^2+(\alpha^{-2}\beta_i-G)\va^2\right]dr}{\displaystyle\int_{r_1}^{r_2}\alpha^{n-1}\va^2\,dr} \leq\\[.3cm] &\!\!\!\! \min_{\va\in \mathcal{H}^1(r_1,r_2)}\!\! \frac{-\alpha^{n-1}\va\dot{\va}\Big|_{r_1}^{r_2} + \displaystyle\int_{r_1}^{r_2}\alpha^{n-1}\left[(\dot{\va})^2+(\alpha^{-2}\beta_{i+1}-G)\va^2\right]dr}{\displaystyle\int_{r_1}^{r_2}\alpha^{n-1}\va^2\,dr} =\overline{\beta}_1^{i+1}.
\end{array}
\]
\qed

\smallskip

\begin{proposition}\label{conditionwarpedmixed}
Consider the mixed eigenvalue problem \eqref{mixequation} for a warped metric. If, for some $i_0\in \mathbb{N}$, it simultaneously satisfies $J\leq 0$ and $\alpha^{-2}\beta_{i_0}\geq G$ for all $r\in (r_1,r_2)$, then all the eigenvalues of the form $\overline{\beta}^i_j$ with $i\geq i_0$ are non-negative. If, in addition, at least one of the two inequalities is strict, then such eigenvalues are positive. 
\end{proposition}

\begin{proof} 
Let $\{\overline{\beta}^{i_0}_{j}\}_{j\in\mathbb{N}}$ be the sequence of all the eigenvalues	of the Sturm-Liouville problem \eqref{mixproblemwarped} with $\beta=\beta_{i_0}$. According to \eqref{rquotientwarped}, the first element of this sequence is characterized by
		


\[
\overline{\beta}_{1}^{i_0}=\min_{\va\in \mathcal{H}^1(r_1,r_2)} \frac{-\alpha^{n-1}\va\dot{\va}\Big|_{r_1}^{r_2} + \displaystyle\int_{r_1}^{r_2}\alpha^{n-1}\left[(\dot{\va})^2+\left(\alpha^{-2}\beta_{i_0}-G\right)\va^2\right]dr}{\displaystyle\int_{r_1}^{r_2}\alpha^{n-1}\va^2\,dr}.
\]

Therefore, assuming that $\va$ reaches such a minimum and using the boundary conditions in (\ref{mixproblemwarped}), yields

\[\begin{array}{rl}
\overline{\beta}_1^{i_0}\Big(\displaystyle\int_{r_1}^{r_2}\alpha^{n-1}\va^2\,dr+\alpha^{n-1}\left(\va^2(r_2)+\va^2(r_1)\right) \Big)\\[.3cm]=-\alpha^{n-1}\left(\va^2(r_2)J(r_2) + \va^2(r_1)J(r_1)\right)\\[.3cm] \qquad + \displaystyle\int_{r_1}^{r_2}\alpha^{n-1}\left[(\dot{\va})^2+\left(\alpha^{-2}\beta_{i_0}-G\right)\va^2\right]dr.
\end{array}
\]

As we can see, the right term of the above equality is non-negative under the hypothesis $J\leq 0$ and $\alpha^{-2}\beta_{i_0}\geq G$; and it is positive if one of previous inequalities is strict. Therefore, $\overline{\beta}_1^{i_0}$ is non-negative or positive accordingly. Finally, the result follows by recalling that

\[
\overline{\beta}^{i_0}_1\leq \overline{\beta}^i_1\leq \overline{\beta}^i_j,
\] 
where we have used Lemma \ref{ineeigenvalues} for the first inequality.
\end{proof}

\begin{remark}{\em 
Previous result is optimal in the sense that, if we have that $J=0$ and $\alpha^{-2}\beta_{i_0}=G$, then zero can be an eigenvalue of \eqref{mixproblemwarped}. In fact, if $\alpha\equiv 1$, the constant function is an eigenfunction of the zero eigenvalue.}
\end{remark}

As we will see in the forthcoming sections, Proposition \ref{conditionwarpedmixed} is particularly interesting in order to obtain results for both rigidity and bifurcation, whenever $J\leq 0$. If $G\leq 0$, it says that all the eigenvalues of the mixed problem are non-negative. If $G>0$, it will restrict the possible zero eigenvalues to a finite family of Sturm-Liouville problems (see Section \ref{sec5.2}).

\section{Rigidity under metric variations}\label{section3}

\subsection{Stating the variational problem}\label{metric}

We will consider the variational approach first presented by York \cite{York}, Gibbons and Hawking \cite{GH}, which has been further studied by Araujo in \cite{A}. In these works, they consider a functional whose critical points are metrics with constant scalar curvature on the manifold $M$ and where the boundary $\partial M$ is composed by a minimal hypersurface $\bord{1}$ and a CMC hypersurface $\bord{2}$.

Let us begin by considering the so-called Gibbons-Hawking-York functional, GHY-functional for short (also known as the total scalar curvature plus total mean curvature functional):

\begin{equation}\label{GHY}
\begin{array}{l}\ghy:\bifspacem\rightarrow \mathbb{R}\nonumber,\qquad g\longmapsto \ghy(g)=\displaystyle\int_M \scal{g}\,\fv{g}+2\displaystyle\int_{\partial M}\mean{g}\,\fvr{g},\end{array}
\end{equation}
where $R(g)$ is the scalar curvature in $M$ and $H(g)$ is the mean curvature on $\partial M$. The first variation of such a functional reads as follows (see Section \ref{appendix}, Equation \eqref{secondghy} in the Appendix for details)

\begin{equation}\label{firstghy}
\def\arraystretch{2.2}
\begin{array}{r>{\displaystyle}l}

\delta \ghy_{g}(h) = & -\int_M\left(R_{ij}-\frac{1}{2}g_{ij}\scal{g}\right)h^{ij} \fv{g}\\
 & - \int_{\partial M} \left(\II_{\alpha\beta}
-\mean{g} g_{\alpha\beta}\right)h^{\alpha\beta}\fvr{g},\\
\end{array}
\end{equation}

\noindent where $R_{ij}$, $h^{ij}$ and $\II_{\alpha\beta}$ are the components for the Ricci tensor, the $(0,2)$-tensor $h\in T_g\bifspacem(=S^k(M))$ and the second fundamental form $\II$ respectively. From here, we deduce the following result

\begin{proposition}

A Riemannian metric $g$ is a critical point for the functional $\ghy$ if and only if the Riemannian manifold $(M,g)$ is Ricci flat and $\partial M$ is totally geodesic.
\end{proposition}

Therefore, the GHY-functional defined in the whole set of Riemannian metrics does not determine the variational problem of our interest. In order to solve this, we will restrict the domain of the functional to a particular subset of Riemannian metrics. For the sake of clearness, such restriction will be performed in two steps. On the first one, we will consider Riemannian metrics satisfying the constraint $\Cab(g)=1$, where:

\begin{align*}\label{defCab}
\Cab(g)=a\int_{M}\fv{g} + b\int_{\bord{2}}\fvr{g_{_2}},
\end{align*}
 with $a,b$ real numbers, $a\geq 0$, and $g_{_2}$ denoting the metric $g$ restricted to $\bord{2}$.

\begin{remark}\label{remark1}
\hfill\break
\begin{itemize}
\item[(i)]{\em
It is clear that the above constraint was deduced from \eqref{c}, since in our problem only $\Sigma_2$ must have constant mean curvature, and not all the boundary. In fact, as it is known, $\int_{ M}\fv{g}$ determines the volume of $M$ while $\int_{\bord{2}}\fvr{g_{_2}}$ determines the area of $\bord{2}$.\medskip

\item[(ii)] Note that if $a=0$ (resp. $b=0$) then $b=1/\int_{\bord{2}}\fvr{g_{_2}}>0$ (resp. $a=1/\int_{ M}\fv{g}>0$). So, it is not restrictive to assume that if $a=0$ then $b=1$; or, in other case, if $b=0$ then $a=1$.}
\end{itemize}
\end{remark}

\smallskip

Let us denote by $\bifspacem_{a,b}$ the space of metrics under the following constraint
\[
\bifspacem_{a,b}:=\{g\in \bifspacem: \,\Cab(g)=1\}.
\]
In order to study the critical points of the GHY-functional restricted to $\bifspace_{a,b}$, the Lagrange multipliers method leads us to the study of the following functional

\begin{equation}\label{deffunclambda}
\biffunm{\lambda}(g)=\ghy{(g)}-\lambda(\Cab(g)-1), \textrm{ for some }\lambda\in\mathbb{R},
\end{equation}
%
whose first variation takes the following form
\[\def\arraystretch{2.2}\begin{array}{r>{\displaystyle}l}
\delta\, (\biffunm{\lambda})_{g}(h) = & \int_{M}\left(R_{ij}-\frac{(\scal{g}+\lambda a)}{2}g_{ij}\right)h^{ij}\fv{g} +\\
 & + \int_{\bord{1}}\left(\II_{\alpha\beta}-\mean{g_{_1}}g_{\alpha\beta}\right)h^{\alpha\beta} \fvr{g_{_1}}+\\
  & + \int_{\bord{2}}\left(\II_{\alpha\beta}-\frac{(2\mean{g_{_2}}+\lambda b)}{2}g_{\alpha\beta}\right)h^{\alpha\beta}\fvr{g_{_2}}
\end{array}
\]
where $g_{_k}$ denotes the metric $g$ restricted to $\bord{k}$. Then, by a classical argument involving the first Bianchi identity, we obtain the following result
\begin{proposition}\label{propcriticos}
A Riemannian metric $g$ on $M^n$ is critical for $\ghy$ restricted to the space $\bifspacem_{a,b}$ if and only if $g$ is Einstein, $\bord{1}$ is totally geodesic, $\bord{2}$ is totally umbilical with constant mean curvature and, if $\scal{g}$ denotes the scalar curvature and $\mean{g_{_2}}$ the mean curvature in $\bord{2}$, then
\begin{equation}\label{relscalmean}
\begin{array}{c}
(n-1)\,b\,\scal{g}=2n\,a\,\mean{g_{_2}}.
\end{array}\end{equation}
\end{proposition}

\smallskip

\smallskip


In the second step, we restrict the domain of $\ghy$ to the space of conformal metrics of $g$ which also lie in $\bifspacem_{a,b}$, that is, we will consider the subset of Riemannian metrics

\[
\mathrm{Conf}_{a,b}(g):=\bifspacem_{a,b}\cap \conf g
\]
where $\conf g$ is defined in \eqref{eq:defconf}.

\begin{remark}\label{remarkbifurcation}{\em
Recall that our aim is to use the fiber bundle version of the Implicit Theorem, and so, a Banach space as $\mathcal{H}^1$ will be enough for our purposes. However, for the study of bifurcation points, we have to deal with the fiber bundle version of the classical result given by Smoller and Wasserman for bifurcation \cite{S-W}. This result requires additional technical conditions as Fredholmness, Palais-Smale, among others, and so, the Sobolev space has not the required regularity. In such a case, it is necessary to restrict even more the conformal class to the H\"older space ${\mathcal C}^{k,\alpha}$ with $k\geq 0$ and $0<\alpha\leq 1$ (see \cite{LPZ2} for detailed studies).}
\end{remark}

\smallskip

The set $\mathrm{Conf}_{a,b}(g)$ is a smooth submanifold of $\conf g$ because it is the set of regular points of the map $C_{a,b}(g)$, being its tangent space $T_g(\mathrm{Conf}_{a,b}(g))$ identified with the set

\begin{equation}\label{tangentspace}
\mathcal{H}^1_{a,b}(M)\!=\!
\left\{f\in \mathcal{H}^1(M): f>0 \,\textrm{ and }\,
\frac{na}{2}\int_{M}f \fv{g}+\frac{(n-1)b}{2}\int_{\bord{2}}f\fvr{g_{_2}}=0\right\}.
\end{equation}
Thus, under this last restriction, the first variation of the functional $\biffunm{\lambda}$ becomes (see \eqref{firstcabfg}, \eqref{firstFfg})

\[\def\arraystretch{2.2}\begin{array}{r>{\displaystyle}l}
\delta (\biffunm{\lambda})_{g}(f)= &\int_{M} \left(\scal{g}-n\frac{(\scal{g}+\lambda a)}{2}\right) f\fv{g} +\\
 & - \int_{\bord{1}}(n-2)\mean{g_{_1}} f \fvr{g_{_1}} +\\
 & - \int_{\bord{2}}\left((n-2) \mean{g_{_2}} - (n-1)\frac{\lambda b}{2}\right) f \fvr{g_{_2}}
\end{array}
\]
with $f\in \mathcal{H}^1_{a,b}(M)$. Therefore, we finally conclude that 

\begin{proposition}\label{criticalpointch}
A Riemannian metric $g$ is a critical point of the functional $\ghy$ restricted to the space $\mathrm{Conf}_{a,b}$ if and only if the scalar curvature $\scal{g}$ is constant on $M$, $\bord{1}$ is a minimal hypersurface ($\mean{g_{_1}}=0$), and $\bord{2}$ has constant mean curvature $\mean{g_{_2}}$.
Under these assumptions, $\scal{g}$ and $\mean{g_{_2}}$ are related by \eqref{relscalmean}.
\end{proposition}

\begin{remark}\label{remarkabsecond}{\em
For prescribed $\scal{g}$ and $H(g_{_2})$, the values of $a$ and $b$ are determined by the equations $\Cab(g)=1$ and (\ref{relscalmean}) (see also Remark \ref{remark1}). Moreover, if $g$ is one of the critical points described in Proposition \ref{criticalpointch}, it will be also critical for the functional $\biffunm{\lambda}$ defined in (\ref{deffunclambda}) with $\lambda$ solving one of the following equations (if $a\neq 0 \neq b$, both equations define the same $\lambda$ assuming (\ref{relscalmean})):

\begin{equation}\label{lambdameanscalar}
b\lambda=H(g_{_2})\frac{2(n-2)}{n-1},\quad \textrm{and}\quad a\lambda=R(g)\frac{n-2}{n}.
\end{equation}
}\end{remark}

\smallskip

This last restriction lead us to a variational problem whose critical points satisfy the desired properties. As we brought to remembrance at the end of Section \ref{sec2.1}, the criteria for rigidity (as well as for bifurcation) make use of the second variation of the functional over critical points. For
%
$g$ a critical point of $\biffun{\lambda}$, with $\lambda$ depending on both, $R(g)$ and $H(g_{_2})$,
the quadratic form associated of such second variation takes the following form on $\mathcal{H}^1_{a,b}(M)$ (see \eqref{cabsecond}, \eqref{secondvariationghy} and the Appendix for details)
{\small
\begin{align*}\label{segunda}
\nonumber\qquad\delta^2 (\biffunm{\lambda})_{g}(f,f)= &\frac{(n-2)(n-1)}{2}\left( \int_{M}\!\left[|\grad{f}|_g^2-\!\frac{\scal{g}}{(n-1)}f^2 \right]\!\fv{g}-\!\int_{\partial M}\!\frac{\mean{g}}{(n-1)}f^2 \fvr{g}\right) \\ = & \frac{(n-2)(n-1)}{2}\left( \int_{M}\!\left[f\Delta_gf-\frac{\scal{g}}{(n-1)}f^2 \right] \!\fv{g}+\!\int_{\partial M}\!\left[-f\partial_nf-\frac{\mean{g}}{(n-1)}f^2\right]\!\fvr{g}\right)
\end{align*}}

\noindent where $|\grad{f}|_g^2$ denotes the squared norm of the gradient of $f$, 
$\Delta_g$ is the Laplace-Beltrami operator and $\partial_n$ is the inward normal derivative (recall that, for the second equality, we have used the first Green identity). Then, we can recover the expression for the second variation and, even more, we can describe it in terms of Fredholm operators by using the following inner product on the space $L^2(M)\cap L^2(\partial M)$

\[
\mathcal{E}(k,f)=\int_M kf \,\fv{g}+\int_{\partial M} kf\,\fvr{g}.\]

In particular, the second variation takes the following form

\begin{align*}
\delta^2 (\biffunm{\lambda})_{g}(f,k) =  \frac{(n-2)(n-1)}{2}\mathcal{E}(J_g(f),k)
\end{align*}
where $J_g$ is a linear elliptic operator given by
\[J_g|_{M}=\Delta_g -\dfrac{\scal{g}}{(n-1)},\qquad\quad J_g|_{\partial M}=-\partial_n-\frac{\mean{g}}{(n-1)} .
\]
From the second Green identity follows that $J_g$ is a self-adjoint operator relative to the $L^2(M)\cap L^2(\partial M)$-inner product. Namely 
\begin{align*}
\mathcal{E}(J_g(f),k)
&=\int_M\left(k\Delta_gf -\dfrac{\scal{g}}{(n-1)}kf\right)\fv{g}
-\int_{\partial M}\left(k\partial_nf+ \frac{\mean{g}}{(n-1)} kf\right)\fvr{g}\\&=\int_M\left(f\Delta_gk -\dfrac{\scal{g}}{(n-1)}kf\right)\fv{g}
-\int_{\partial M}\left(f\partial_nk+ \frac{\mean{g}}{(n-1)} kf\right)\fvr{g}\\&=
\mathcal{E}(f,J_g(k)).
\end{align*}

Therefore, $J_g$ coincide with the Jacobi operator associated to $\delta^2(\biffun{\lambda})_g$. Then, for our studies on rigidity, we have to study the spectrum of $J_g$ which lead us to the mixed eigenvalue problem

\begin{eqnarray}\label{bifvariation}
\left\{
\begin{array}{ccc}\Delta_g f-\dfrac{\scal{g}}{(n-1)}f=\overline{\mu} f\quad \textrm{in } M,\\\\
-\partial_n f-\dfrac{\mean{g}}{(n-1)}f=\overline{\mu} f
\quad \textrm{on } \partial M
\end{array}
\right.
\end{eqnarray}

\noindent with $f\in \mathcal{H}^1_{a,b}(M)$.

\begin{remark}\label{remarkadmissible}{\em
		It is worth pointing out that in spite of the general case studied in Section \ref{mixed}, here the eigenfunctions should belong to $\mathcal{H}^1_{a,b}(M)$. This means that not all the eigenvalues of previous mixed problem are considered, but instead a subfamily of them. We will call {\em admissible} to an eigenvalue of \eqref{bifvariation} whose associated eigenfunction satisfies the integral condition in \eqref{tangentspace}.   
		}
\end{remark}

\smallskip

\subsection{General rigidity results under metric variations}\label{gere}

The lower bound for eigenvalues obtained in Proposition \ref{lowboundg} allows us to give a general result for rigidity associated to metric variations by giving a very simple and geometrical argument.

\begin{theorem}\label{tege}
Let $M$ be a compact manifold with boundary $\partial M=\Sigma_1\cup \Sigma_2$. Consider $\{g_{\lambda}\}_{\lambda\in I}$ a family of metrics which are critical points for the functional $\biffunm{\lambda}$ restricted to $\mathrm{Conf}_{a,b}$ for some $a=a(\lambda)$ and $b=b(\lambda)$; that is, metrics with constant scalar curvature in $M$ and with zero mean curvature on $\Sigma_1$ and constant mean curvature on $\Sigma_2$ 
. If 
$R(g_{\lambda_*}),H((g_{\lambda_*})_{_2})\leq 0$ for some $\lambda_*\in I$, being one of these inequalities strict, then the family $\{g_{\lambda}\}_{\lambda\in I}$ is locally rigid at $\lambda_*$.
\end{theorem}
\begin{proof} Let us study the spectrum of $J_{g_\lambda}$ for $\lambda=\lambda_*$. As we have pointed out in the previous section, the spectrum of such a linear map is related with the mixed eigenvalue problem detailed on (\ref{bifvariation}) with $g=g_{\lambda_*}$. Now, Proposition \ref{lowboundg} ensures that
all the eigenvalues for the Jacobi operator are positive and $g_{\lambda_*}$ is a non-degenerate critical point. Then, the result follows from the fiber bundle version of the Implicit Function Theorem (see \cite[Appendix A]{LPZ2}).
\end{proof}

\section{Metric rigidity for foliations in warped product spaces} \label{wm}

Warped product spaces determines a foliation of the entire space by CMC-hypersurfaces known as slices (the natural leafs of the product). Such a foliation can be related with a family of metrics defined on a (fixed) compact subspace with the structure of a manifold with boundary. In particular, under the additional assumptions of constant scalar curvature on the space and the minimality of one of the leafs, our previous results on metric rigidity are applicable. As the construction of the metrics from the foliation is reversible, such a rigidity will be interpretable as metric rigidity for the natural foliation on warped spaces.

\smallskip

Let us consider \begin{equation}\label{wpf}(M^n,g)=((r_1,r_2)\times P^{n-1}, dr^2+\alpha(r)^2g^P)\end{equation} an $n$-dimensional warped product space with constant scalar curvature where $P$ is an $(n-1)$-dimensional closed (compact without boundary) manifold. For each $\gamma\in [r_1,r_2)$, let us denote by $\Sigma_{\gamma}=\{\gamma\}\times P$ a leaf of the natural foliation by CMC-hypersurfaces of the warped product. Along this section, we will assume that the metric is extensible to the slice $\Sigma_1:=\Sigma_{r_1}$ and $\Sigma_1$ 
is a minimal hypersurface. Moreover, we will also assume that $r_2\in \mathbb{R}\cup\{\infty\}$ (the case with $\{r_2\}\times P$ minimal is analogous). 

For $\gamma\in (r_1,r_2)$, consider the slab $\opens_{\gamma}=[r_1,\gamma]\times P$ a compact manifold whose boundary is composed by the leafs $\Sigma_1$ and $\Sigma_\gamma$ of the foliation (see Figure \ref{fig2}), being the (constant) mean curvature of the latter given by
\begin{align}\label{CMC}
H(\gamma)=-(n-1)\frac{\dot{\alpha}(\gamma)}{\alpha(\gamma)}
\end{align}
(recall that, as $\Sigma_1$ is minimal, previous equation implies that $\dot{\alpha}(r_1)=0$). 

 By a standard procedure, for a fixed $r_0\in (r_1,r_2)$ we can define a diffeomorphism $\Psi_\gamma:\opens_{\gamma}\rightarrow \opens_{r_0}$ preserving the orientation of $\partial/\partial r$ (for $\gamma=r_0$, such a diffeomorphism is just the identity). Then, we define a family of metrics $\{g_{\gamma}\}_{\gamma\in (r_1,r_2)}$ on $\opens_{r_0}$ given by the push forward of the metric $g$ of $\opens_{\gamma}$ on $\Omega_{r_0}$, i.e., 
\begin{equation}\label{defmetric}
g_{\gamma}=(\Psi_{\gamma})_*g.
\end{equation}
\vspace*{0.1cm}

\begin{figure}[htbp]
	\includegraphics[width=12cm, height=9cm]{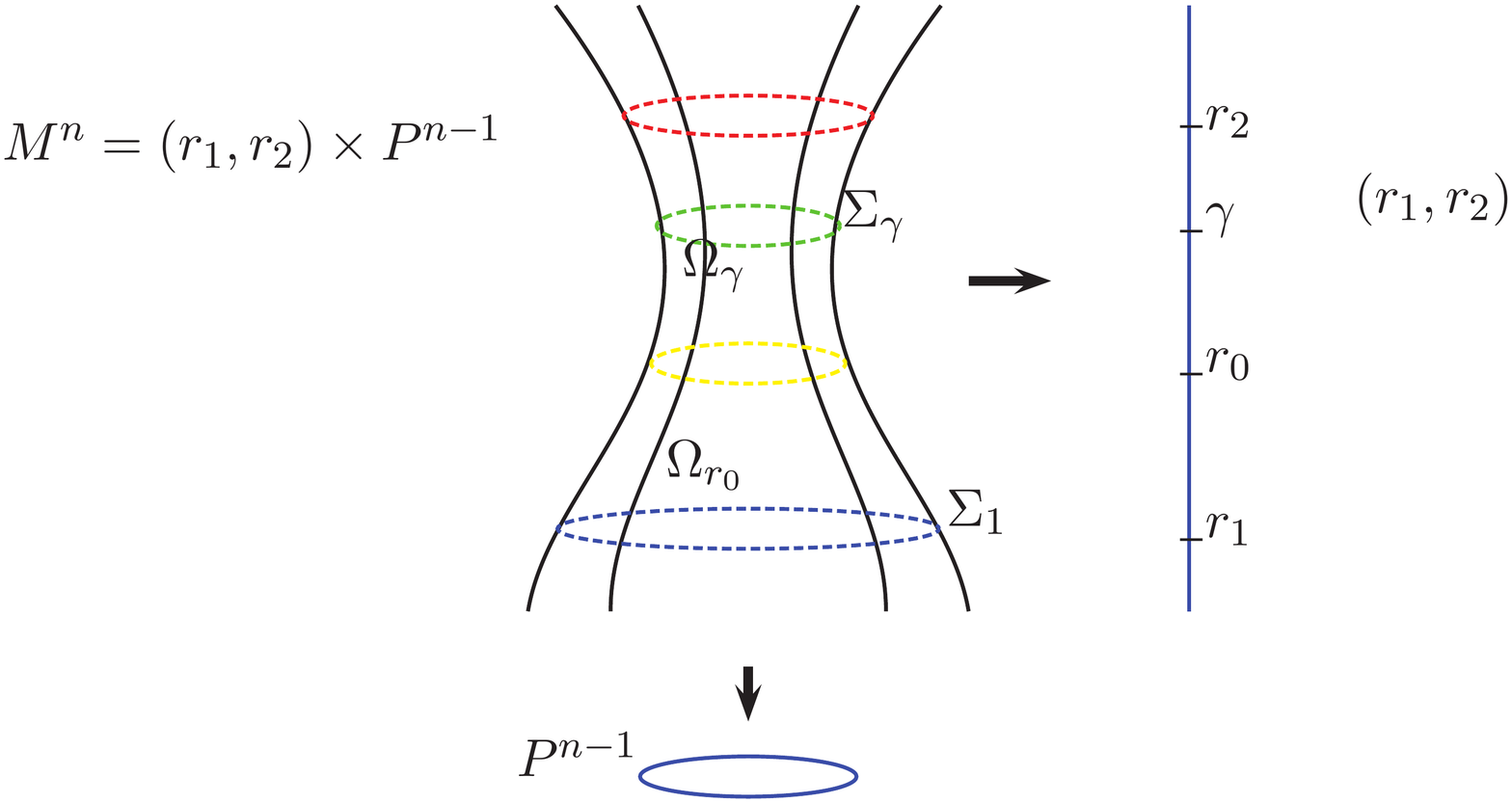}
\caption{\label{fig2}Construction of $\Omega_{\gamma}$'s.}
\end{figure}

\begin{remark}\label{remarkiden}{\em
As it is clear from the above construction, both $(\opens_{r_0},g_{\gamma})$ and $(\opens_{\gamma},g)$ are isometric under $\Psi_\gamma$, and then, they share the same geometrical properties.
  So that, we can make the computations with the latter, which will be simpler in practical cases. Moreover, this approach shows the independence on $r_0$ of our results.
  }
\end{remark}

By taking into account the properties described for the slab $\Omega_{\gamma}$ and previous remark, we have that
 $g_{\gamma}$ is a critical point for the functional $\biffunm{\lambda(\gamma)}$, being $\lambda:(r_1,r_2)\rightarrow \R$ a function given implicitly on (\ref{lambdameanscalar}) (recall that $a$ and $b$ are determined by $R(g_{\gamma})$ and $H((g_{\gamma})_{_2})\equiv H(\gamma)$, see Remark \ref{remarkabsecond}). So, we are ready to apply Theorem \ref{tege} to the family $\{g_{\gamma}\}_{\gamma\in (r_1,r_2)}$ defined on $\Omega_{r_0}$.

However, it is interesting to present the result on this section, not as a property for the family of metrics, but for the family of compact sets $\{\Omega_{\gamma}\}_{\gamma\in (r_1,r_2)}$. In this sense, we will say that the family $\{\Omega_{\gamma}\}_{\gamma\in (r_1,r_2)}$ has a bifurcation point on $\gamma_*$ if there exist sequences $\{f_n\}_n,\{\gamma_n\}_n$ with $f_n\in \mathcal{H}^1(\Omega_{\gamma_n})$, $\gamma_n\in (r_1,r_2)$ such that:
\begin{itemize}
	\item[(1)] $\lim\limits_{n\rightarrow \infty}\gamma_n=\gamma_*$ and $\lim\limits_{n\rightarrow \infty}g_n=g_{\lambda_*}$ where $g_n:=f_ng$.
	\item[(2)] For all $n$, the metric $g_n$ defined on $\Omega_{\gamma_n}$ determines constant scalar curvature $R(g)$; and makes $\Sigma_1=\{r_1\}\times P$ a minimal hypersurface and $\Sigma_{\gamma}=\{\gamma\}\times P$ a hypersurface with constant mean curvature $H(\gamma)$.
	\item[(3)] For all $n$, $f_n\not\equiv 1$.
\end{itemize} 	
	
\noindent If $\gamma_*$ is {\em not} a bifurcation point, we will say that the family $\{\Omega_{\gamma}\}_{\gamma}$ is locally metrically rigid at $\gamma_*$. 

\smallskip

As it is clear from construction, rigidity and bifurcation of the family $\{\Omega_{\gamma}\}_{\gamma}$ are equivalent to rigidity and bifurcation on the associated family of metrics $\{g_{\gamma}\}_{\gamma}$. Then, Theorem \ref{tege} lead us to

\begin{corolario}\label{rigidwarpedmetric2}
Let $(M,g)$ be a warped product space as in \eqref{wpf} with constant scalar curvature $R(g)$ and being $\Sigma_1:=\{r_1\}\times P$ a minimal hypersurface. Let us consider $\{\Omega_{\gamma}\}_{\gamma\in (r_1,r_2)}$ a family of compact sets defined by $\Omega_{\gamma}=[r_1,\gamma]\times P$. If $R(g),H(\gamma_*)\leq 0$ for some $\gamma_*\in (r_1,r_2)$, being at least one of these inequalities strict, then $\gamma_*$ is a metrically rigid point for the family of sets $\{\Omega_{\gamma}\}_{\gamma}$.
\end{corolario}

\smallskip

\subsection{Some examples of rigid warped spaces}\label{ADS}

When $(M,g)$ is a warped metric as described in \eqref{wpf}, the scalar curvatures of both $g$ and $g^P$ are related by 

\begin{equation}\label{eq:scalarcurvatures}
R(g)=\frac{(n-1)}{\alpha^2}\left(R(g^P)-\left((n-2)\dot{\alpha}^2 + 2\ddot{\alpha}\alpha\right)\right).
\end{equation}

As it is clear, whenever the warping function $\alpha$ and the scalar curvature of the manifold $P$ $R(g^P)$ are constant, the scalar curvature of the warped space $R(g)$ is also constant. Moreover, all the
slices in such a case are minimal, so $H(\gamma)\equiv 0$. Therefore, if $R(g^P)<0$ by Corollary \ref{rigidwarpedmetric2} we have that the family $\{\Omega_{\gamma}\}_{\gamma}$ is metrically rigid. 

However, we can obtain others non-trivial results. Let us assume that $n=3$, $R(g^P)$ is constant and $R(g)$ is
a non-positive constant. Then, by writing $R(g)=-6E$ for some non-negative constant $E$, we can deduce
\begin{equation}\label{NC}
\dot{\alpha}(r)^2=R(g^P)-\frac{2K}{\alpha(r)}+E\alpha(r)^2,\quad\textrm{whenever $K$ to be a constant}.
\end{equation} 
Indeed, from \eqref{eq:scalarcurvatures} we have that $R(g^P)+3E\alpha^2=\dot{\alpha}^2+2\ddot{\alpha}\alpha$ multiplying by $\dot{\alpha}$ we obtain 
\[\left(R(g^P)+3E\alpha^2\right)\dot{\alpha}=\dot{\alpha}^3+2\ddot{\alpha}\dot{\alpha}\alpha=\stackrel{\small{\bullet}}{(\alpha\dot{\alpha}^2)}\]
and integrating we have
\[R(g^P)\alpha+E\alpha^3=\alpha\dot{\alpha}^2+2K,\]
where $K$ is an integration constant, and \eqref{NC} follows.

Next, taking the variable change 

\begin{equation}\label{eq:cambiovariable}
s\equiv\alpha(r)
\end{equation}
the warped space becomes

\begin{equation}\label{eq:ejemplos}
\left((\widehat{s},s_2)\times P, \frac{1}{R(g^P)-\frac{2K}{s}+Es^2}ds^2+s^2g^P\right)
\end{equation}
where $\widehat{s}\geq 0$ is a zero for $R(g^P)-2K/s+Es^2$ and $s_2\in \mathbb{R}\cup\{\infty\}$. As we can see from \eqref{CMC} and \eqref{eq:cambiovariable}, the mean curvature function $H(\gamma)$ takes the form

\[
H(\gamma)=-2\frac{\sqrt{R(g^P)-2K/\gamma+E\gamma^2}}{\gamma}.
\]

So, it is negative except at $\gamma=\widehat{s}$ where is zero. Thus, the metric space described in \eqref{eq:ejemplos} satisfies that $R(g),H(\gamma)\leq 0$, being the inequality for the mean curvature strict for $\gamma>\widehat{s}$. In conclusion, Corollary \ref{rigidwarpedmetric2} lead us to

\begin{theorem}\label{three-rigid}
Let $(M^3,g)$ be a warped metric as in \eqref{wpf} with fiber of constant scalar curvature. If $R(g)$ is a non-positive constant, $\alpha$ is monotonically increasing and $\dot{\alpha}(r_1)=0$, then the family of compact sets $\{\Omega_{\gamma}\}_{\gamma}$ is metrically rigid.  
\end{theorem}


\smallskip

There are several examples with $R(g^P)$ constant. For instance, if we look for $R(g^P)=-1$, we can take $P=\mathbb{H}^2/\Gamma$ where $\mathbb{H}^2$ is the $2$-dimensional hyperbolic space and $\Gamma$ a co-compact group of isometries defined on it. For $R(g^P)=0$, we can take $P=\mathbb{T}^2$ the $2$-dimensional flat torus. The case for $R(g^P)=1$, where $P$ is naturally the $2$-sphere $\mathbb{S}^2$, is specially interesting due to its relation with general relativity. In fact, for $E\geq 0$, the space 

\begin{equation} 
\left((\widehat{s},\infty)\times \mathbb{S}^2,g_{K,E}\right),\qquad g_{K,E}:=\frac{1}{\sqrt{1-\frac{2K}{s}+Es^2}}ds^2+s^2g^{\mathbb{S}^2}
\end{equation}
represents the spatial slice of the Anti-de Sitter Schwarzschild spacetime. The constant $K$ denotes here the mass of the system, and so, it is considered positive. $E$ is the so-called cosmological constant. In the particular case where $E=0$, we obtain the Schwarzschild model.

As we can see in \cite[Chapter 6]{Wa}, the Schwarzschild models appear as a solution of the Einstein equations under some mild hypotheses. On the one hand, it is considered the case of {\em vacuum} Einstein equations, which lead us to the so-called Einstein metrics (metrics whose Ricci tensor is proportional to themselves). On the other hand, it is assumed that the spacetime is static.

The minimal fibre $\Sigma_1=\{\widehat{s}\}\times \mathbb{S}^2$ has a nice interpretation from the viewpoint of relativity. It determines the so-called marginally outer trapped surface (MOTS for short) which represents the boundary of the black hole contained in the model. That is, $\Sigma_1$ is a two-sided surface where the outgoing lightlike rays emanating from it collapse marginally (see \cite{AndMarSim} and the references therein for general background on MOTS).
 
In particular, the previous results and the definition of metric rigidity allow us to deduce that small variations of the metric inside the conformal class, fixing the scalar curvature $R(g_{K,E})$ and making $\Sigma_{1}$ minimal cannot have slices $\{\gamma\}\times \mathbb{S}^2$ with constant mean curvature $H(\gamma)$. In the context of relativity, the results imply that, once it is fixed the causal structure of the model (which is determined by the conformal class of the metric), the spatial fibers of Anti-de Sitter Schwarzschild spacetime are locally unique in the family of Einstein metrics preserving MOTS and making $\{\gamma\}\times \mathbb{S}^2$ a CMC surface with prescribed mean curvature.

\subsection{Warped product manifolds with positive constant scalar curvature}\label{sec5.2}

Notice, in the previous rigidity results, we always require that the mean curvature and the scalar curvature are non-positive. However, if we look at Proposition \ref{conditionwarpedmixed}, the result gives some additional information in cases where the scalar curvature is positive. In this last section, we will show a simple method to obtain both rigidity and bifurcation results, when the mean curvature is negative but the scalar curvature is positive.

\smallskip 

Let us consider $(M,g)$ a warped metric space as in \eqref{wpf} and assume that $R(g)$ is a positive constant. We will also assume that the warping function $\alpha$ is non-decreasing (i.e. $\dot{\alpha}\geq 0$), being zero in $r_1$. In particular, and recalling the expression for $H(\gamma)$ in \eqref{CMC}, the mean curvature of the slice $\{\gamma\}\times P$ is nonpositive for $\gamma\in (r_1,r_2)$ and zero in $\gamma=r_1$. Our aim is to study the rigid character of the family of compact sets $\{\Omega_{\gamma}\}_\gamma$ and, to do so, we have to study for each $\gamma$ the spectrum of the mixed eigenvalue problem \eqref{bifvariation} with $M=\Omega_{\gamma}$. 

Let us recall two important facts. First, that the eigenfunctions $f$ of previous mixed problem should belong to the space $\mathcal{H}^1_{a,b}(\Omega_{\gamma})$, where $a$ and $b$ are determined by $R(g)$ and $H(\gamma)$ (Remark \ref{remarkabsecond}). Secondly, and recalling Theorem \ref{teomixedproblemw}, that any eigenfunction $f$ of the problem can be written as the product of two functions $f_{P}$ and $\fa$, where the former is an eigenfunction of the Laplacian on $(P,g^P)$ and the latter is an eigenfunction of the following Sturm-Liouville problem 

\begin{equation}\label{sturmliuwarped}
\begin{array}{rl}
-\dfrac{1}{\alpha^{n-1}}\partial_r\left(\alpha^{n-1}\partial_r\fa\right)+\dfrac{1}{\alpha^{2}}\mu\fa=& \left(R(g) + \overline{\mu}\right)\fa,\\\\
-\dot{\fa}(r_1)=&\overline{\mu}\fa(r_1),\\\\
\dot{\fa}(\gamma)=&\left(H(\gamma)+\overline{\mu}\right)\fa(\gamma).
\end{array}
\end{equation}
Here, $\mu$ denotes the eigenvalue associated to $f_{P}$. The eigenvalues of previous Sturm-Liouville problem will be denoted, in accordance with Section \ref{mixedwarpedsection}, as $\{\overline{\mu}_j^i(\gamma)\}$. 

As $f=\fa f_{P}\in \mathcal{H}^1_{a,b}(\Omega_{\gamma})$, we deduce that:
\begin{equation}\label{conditionadmissible}
\left(\frac{na}{2}\int_{r_1}^\gamma\alpha^{n-1}(r)\fa(r)\,dr + \frac{(n-1)b}{2}\fa(\gamma)\alpha^{n-1}(\gamma)\right)\int_{P}f_{P}d\sigma_P=0,
\end{equation}
where $d\sigma_{P}$ denotes the volume form on $(P,g^{P})$. The second integral of previous expression is zero if and only if $f_{P}$ is an eigenfunction associated to $\mu_i$ with $i\neq 0$. When $i=0$, a restriction on the function $\fa$ appears. Therefore, any eigenvalue for \eqref{sturmliuwarped} with $\mu=\mu_i$ and $i\geq 1$ is admissible in the sense described in Remark \ref{remarkadmissible}. However, when $\mu=0$, we have to check additionally that the associated eigenfunction $\fa$ satisfies an integral condition derived from \eqref{conditionadmissible}. 

\smallskip 

Let us assume that $\alpha$ is bounded from above. Then, taking into account that $\displaystyle{\rm lim}_{i\rightarrow\infty}\mu_i=\infty$ and $\alpha$ is non-decreasing, we can find $i_0$ such that for $i> i_0$, 

\[\mu_i \alpha(r)^{-2}\geq \mu_i\alpha(r_2)^{-2} > R(g)\] 
(here, we denote by $\alpha(r_2)={\rm lim}_{r\rightarrow r_2}\alpha(r)$). So, as $H(\gamma)\leq 0$, Proposition \ref{conditionwarpedmixed} ensures that all the eigenvalues $\{\overline{\mu}_{j}^i(\gamma)\}$ with $i> i_0$ are positive. Therefore, we only have to study when $\{\overline{\mu}_j^i(\gamma)\}$ is zero for $0\leq i\leq i_0$. This study can be performed by using numerical approximations in the following way: For each $i=0,1,\dots,i_0$ and any $\gamma\in (r_1,r_2)$, consider the differential equation \eqref{sturmliuwarped} with $\overline{\mu}=0$ and initial conditions\footnote{Observe that $\fa(\gamma)$ cannot be zero as, otherwise, $\dot{\fa}(\gamma)=0$ and so $\fa$ is constantly zero.} $\fa(\gamma)=1$ and $\dot{\fa}(\gamma)=H(\gamma)$. Then, two possibilities arise:
 \begin{itemize}
 	\item[(i)] $\dot{\fa}(r_1)\neq 0$, and then, $\fa$ does not satisfy the initial condition in \eqref{sturmliuwarped} with $\overline{\mu}=0$. In particular, we deduce that zero is not an eigenvalue in this case. 
 	
 	\item[(ii)] $\dot{\fa}(r_1)=0$, and so, zero is an eigenvalue for the mixed problem. In this case, and using the continuous dependence of the eigenvalues regarding $\gamma$, we can make a local numerical analysis to determine if such an eigenvalue changes its sign around $\gamma$. 
 \end{itemize}
 
 If for all $0\leq i\leq i_0$ we are in the first case, we have that $\gamma$ is a (local) rigid point. Otherwise, we have to analyze the cases where the zero eigenvalue appears in order to determine if the Morse index varies, and so, if $\gamma$ is a bifurcation point (recall Remark \ref{remarkbifurcation} about bifurcation results). In this second case, we have to take special attention when $i=0$ in order to ensure that the eigenvalue is admissible.

\section{Appendix}\label{appendix}
In this subsection we will include all the basic computations needed in order to compute the first and second variation of the functional $\biffunm{\lambda}$. Such computations have been already appeared elsewhere (see for instance \cite{Levi,A}), but we will include them here for the sake of completeness.

Even if we are interested only in conformal variations, the computation of the first variation of the different elements conforming $\biffunm{\lambda}$ will be done for a more general family of variations.
Let $S^k(M)$ the space of all symmetric $(0,2)$ tensors of class $C^k(M)$ with $k\geq2$. We consider a metric variation $g:\cone\times(-\epsilon,\epsilon)\rightarrow \mathbb{M}$ defined by:

\[
g(h,t)=g+th
\]
where $\cone$ is the open cone of $S^k(M)$ consisting of all Riemannian metrics on $M$ such that for all $g\in\cone$ the tangent space $T_g\cone$ is identified with the Banach space $S^k(M)$.
By compactness, for $\vert t\vert$ sufficiently small, $g(h,t)$ is in $\cone$.
\begin{convention}
Henceforth, all the elements associated to the metric $g(h,t)$ will be denoted as functions of $t$, assuming that the metric variation $h$ is fixed from the beginning. So, elements as the metric itself, the scalar and mean curvature (among others) will be denoted by $g(t), \scal{t}$ and $\mean{t}$ respectively.
\end{convention}

\smallskip

For the first variation of $\Cab$, just recall that

\begin{align}\label{dv}
\delta\,\fv{t}=\frac{1}{2}g^{ij}h_{ij}\,\fv{g},
\end{align}
and so

\begin{equation}\label{cabfirstdev}
\delta\,\Cab(t)=\frac{a}{2}\int_{M} g^{ij}h_{ij}\;\fv{g} + \frac{b}{2}\int_{\bord{2}} g^{\alpha\beta}h_{\alpha\beta}\;\fvr{g_2}.
\end{equation}
Now, we will focus on the first variation of the GHY-functional (recall the definition on \eqref{GHY}). From basic computations, we have that

\begin{equation}\label{ghyfirst1}
\def\arraystretch{2.2}
\begin{array}{r>{\displaystyle}l}
\delta\,\ghy(t)= & \int_{M} \delta\,\left(\scal{t}\fv{t}\right) + \int_{\partial M} \delta\,\left(\mean{t}\fvr{t}\right)\\ = & \int_{M} (\delta\,\scal{t})\fvr{g} + \scal{g}(\delta\,\fv{t})+\\ &+\int_{\partial M} (\delta\;\mean{t})\fvr{g} + \mean{g}(\delta\, \fvr{t}) \end{array}
\end{equation}

and so, we need to compute the first variation of both the scalar curvature and the mean curvature. For the first one, let us consider a point $p\in M$ and a local normal coordinate system $(x_i)$ centred on $p$. Recall that in such coordinates, at the point $p$ we have:

\[
g_{ij}=\delta_{ij},\quad \partial_{k}g_{ij}=0\quad\hbox{and}\quad \nabla_i=\partial_i
\]
for $1\leq i,j,k\leq n$, where $\delta_{ij}$ represents the Kronecker delta, that is, $\delta_{ij}=1$ if $i=j$ and $\delta_{ij}=0$ if $i\neq j$.  In particular, the variation of the Christoffel symbols becomes\footnote{Observe that the variation of the Levi-Civita connection $\delta \nabla$ is a tensor (in spite of what happens with $\nabla$), and so, a special coordinate system can be considered for its computation.}

\begin{equation}\label{varchris}
\delta\, \Gamma_{ij}^k(t)=\frac{1}{2}g^{kl}\left(\nabla_jh_{il}+\nabla_ih_{jl}-\nabla_lh_{ij}\right).
\end{equation}
 From the definition of the scalar curvature, we have that:
 \[
 \delta\, \scal{t}=(\delta g^{ij}(t))R_{ij} + g^{ij}(\delta\,\tricc{t}{ij}).
 \]
 In one hand, one easily obtain that
\[
 \delta\, g_{ij}(t)=h_{ij}\quad\textrm{and} \quad\delta g^{ij}(t)=-h^{ij}.
\]
On the other hand, the variation of the Ricci tensor has the following expression
\[\def\arraystretch{1.5}\begin{array}{rl}
\delta\,\tricc{t}{ij} = & \delta\, \nabla_l\, \Gamma_{ij}^l(t)-\delta\, \nabla_i\,\Gamma_{jl}^l(t)+\delta\,\left(\Gamma_{ij}^u\Gamma_{lu}^l-\,\Gamma_{lj}^u\,\Gamma_{iu}^l\right)(t)\\ = & \nabla_l\, \delta\Gamma_{ij}^l(t)-\nabla_i\,\delta \Gamma_{jl}^l(t)\\ = &  \frac{1}{2}g^{lm}\left(\nabla_l\nabla_jh_{im}+\nabla_l\nabla_ih_{jm}-\nabla_l\nabla_mh_{ij}-\nabla_i\nabla_jh_{lm}\right)\\  = &\nabla_{i}\left( \nabla_{j}h^{ij}-\nabla_{i}h^l_{l}\right)
\end{array}
\]
where we have used that $\delta\,\left(\Gamma_{ij}^u\Gamma_{lu}^l-\,\Gamma_{lj}^u\,\Gamma_{iu}^l\right)(t)=0$, as the Christoffel symbols vanish for $t=0$, the formulae (\ref{varchris}) and the fact that $g^{lm}\left(\nabla_lh_{jm}-\nabla_mh_{jl}\right)=0$ for $1\leq l,m\leq n$. In particular, the variation of the scalar curvature takes the following form:

\begin{equation}\label{scalarfirst}
\delta\,\scal{t}=-h^{ij}R_{ij} + \nabla_{i}\left( \nabla_{j}h^{ij}-\nabla_{i}h^l_{l}\right).
\end{equation}
Summarizing, the first two elements on the right of \eqref{ghyfirst1} bearing in mind \eqref{dv} and  \eqref{scalarfirst}, we obtain:
\[\def\arraystretch{2.2}\begin{array}{rl}
\displaystyle\int_{M} (\delta\,\scal{t})\fv{g} + \scal{g}(\delta\,\fv{t}) = & \displaystyle\int_M\left(-h^{ij}R_{ij}+\frac{1}{2}g^{ij}h_{ij}\scal{g}\right)\fv{g}\\ & +\displaystyle\int_{M}\nabla_i\left(\nabla_jh^{ij}-\nabla^ih_{m}^{m}\right)\fv{g} \end{array}
\]
or, by using the Divergence theorem,
\begin{equation}\label{firstvariationscalar}\def\arraystretch{2.2}\begin{array}{r>{\displaystyle}l}
\delta\, \displaystyle\int_{M} \scal{t}\fv{t}= &  -\int_M\left(h^{ij}R_{ij}-\frac{1}{2}g^{ij}h_{ij}\scal{g}\right)\fv{g}\\ = & -\int_{\partial M}N^i\left(\nabla^jh_{ij}-\nabla_ih_{l}^{l}\right)\fvr{g}
\end{array}
\end{equation}
where $N$ denotes the inward normal vector to the hypersurface $\partial M$.

Next, we need to compute the second term in \eqref{ghyfirst1}. As the point $p\in \partial M$, let us consider a different coordinate system. Take $(x_1,\dots,x_{n-1})$ a normal coordinate system associated to the boundary $\partial M$, endowed with the metric induced by $g$, and denote by $\gamma_{(x_1,\dots,x_{n})}$ the geodesic starting at $(x_1,\dots,x_{n-1})$ with direction $N$. Observe that $(x_1,\dots,x_n):=\gamma_{(x_1,\dots,x_{n-1})}(x_n)$ defines a coordinate system for $M$ around $p$ in which the metric $g$ takes the following form:
\[g=dx^2_{n}+g_{\alpha\beta}dx^{\alpha}dx^{\beta}.
\]
Moreover, as $(x_{\alpha})$ is a normal coordinate system for $\partial M$ around $p$ and $N=\partial_{n}$,

\[
\Gamma_{\alpha\beta}^n=\II_{\alpha\beta}.
\]
%

\smallskip

Now, we are ready to compute the first variation for the mean curvature. Recall that, from definition,
\[
\mean{t}=g^{\alpha\beta}(t)\II_{\alpha\beta}(t),
\]
and so,
\begin{equation}\label{meancurvature1}
\delta\, \mean{t}=(\delta\, g^{\alpha\beta}(t))\II_{\alpha\beta} + g^{\alpha\beta}(\delta\, \second{t}{\alpha\beta}).
\end{equation}
For the variation of the second fundamental form, define $\nu:\partial M\times (-\epsilon,\epsilon)\rightarrow (T \partial M)^{\bot}$ an unitary inward normal vector of the tangent space for the metric $g(t)$. In particular, it is not restrictive to assume that, for $\nu(0)=N=\partial_{n}$. Then,

\begin{equation}\label{secondvar1}\def\arraystretch{1.5}\begin{array}{rl}
\delta\,\second{t}{\alpha\beta}= &
\delta\, g_{ij}(t)\nu^i(t)\Gamma_{\alpha\beta}^{j}\\
= & h_{nn}\Gamma_{\alpha\beta}^{n} + (\delta \nu(t))^n\Gamma_{\alpha\beta}^n + \delta \Gamma_{\alpha\beta}^n.
\end{array}
\end{equation}
Now, taking into account that $g_{ij}(t)\nu^i(t)\nu^j(t)=1$ and $g_{i\alpha}(t)\nu^i(t)=0$ for all $t$, we obtain that:
\begin{align*}\label{varnu}
(\delta \nu(t))^n=-\frac{1}{2}h_{nn},
\end{align*}
which, together with \eqref{varchris}, yields the following expression for \eqref{secondvar1}:
\[\begin{array}{rl}
\delta \second{t}{\alpha\beta}= & \frac{1}{2}h^{nn}\Gamma_{\alpha\beta}^n+\frac{1}{2}\left(\nabla_\alpha h_{\beta n}+\nabla_{\beta}h_{\alpha n}-\nabla_nh_{\alpha\beta}\right).
\end{array}\]
Hence, \eqref{meancurvature1} becomes
\begin{equation}\label{firstmean}
\def\arraystretch{2.2}
\begin{array}{r>{\displaystyle}l}
\delta\mean{t}= &-h^{\alpha\beta}\II_{\alpha\beta} + \frac{1}{2}h^{nn}g^{\alpha\beta}\Gamma_{\alpha\beta}^n+ \frac{1}{2}g^{\alpha\beta}\left(\nabla_\alpha h_{\beta n}+\nabla_{\beta}h_{\alpha n}-\nabla_nh_{\alpha\beta} \right)\\
= & -h^{\alpha\beta}\II_{\alpha\beta}+\frac{1}{2}h^{nn}\mean{g}+\nabla^{\alpha}h_{\alpha n}-\frac{1}{2}\nabla_{n}h_{\alpha}^{\alpha},
\end{array}\end{equation}
 and so, using \eqref{firstmean} and the fact that $\delta \fvr{t}=(1/2)g^{\alpha\beta}h_{\alpha\beta}\fvr{g}$ we get
{\small \begin{align*}
 \delta\!\int_{\partial M}\! \mean{t}\fvr{t}\!= & \int_{\partial M}\! \left(-h^{\alpha\beta}\II_{\alpha\beta}+\frac{1}{2}h^{nn}\mean{g}+\nabla^{\alpha}h_{\alpha n}-\frac{1}{2}\nabla_{n}h_{\alpha}^{\alpha}+\frac{1}{2}\mean{g} g^{\alpha\beta}h_{\alpha\beta}\!\right)\fvr{g}\nonumber\\
 = & \int_{\partial M} \left[\left(-\II_{\alpha\beta}-\frac{1}{2}\mean{g} g_{\alpha\beta}\right)h^{\alpha\beta}+ \frac{1}{2}h^{nn}\mean{g}\right]\, \fvr{g}\\
  & +\int_{\partial M}\left( \nabla^{\alpha}h_{\alpha n} - \frac{1}{2}\nabla_{n}h_{\alpha}^{\alpha} \right) \fvr{g}.\nonumber
 \end{align*}}
Then, using this identity, (\ref{firstvariationscalar}) and the fact that in our coordinates $\nu=\partial_n$, we deduce that:
\[\def\arraystretch{2.2}\begin{array}{r>{\displaystyle}l}
\delta\,\ghy = & \delta \int_{M} \scal{t}\fv{t} + 2 \delta\int_{\partial M} \mean{t}\fvr{t}  \\
 = & -\int_M\left(h^{ij}R_{ij}-\frac{1}{2}h^{ij}g_{ij} \scal{g}\right)\fv{g} \\
  & + \int_{\partial M} \left[-\left(2\II_{\alpha\beta}-\mean{g} g_{\alpha\beta}\right)h^{\alpha\beta}+ h^{nn}\mean{g}\,\right]\fvr{g} + \int_{\partial M} \left(\nabla^{\alpha} h_{\alpha n}\right)\,\fvr{g}.
\end{array}
\]
Finally, taking into account that,
\[
\nabla^{\alpha} h_{\alpha n}=D_{\alpha}h_{n}^\alpha + h^{\alpha\beta}\II_{\alpha\beta}-h^{nn}\mean{g},
\]
where $D$ denotes the induced connection on $\partial M$, and the fact that
\[
\int_{\partial M} D_{\alpha}h_{n}^{\alpha}\,\fvr{g}=0
\]
from the Divergence theorem (recall that $\partial M$ is closed), we finally obtain the following expression for the first variation of the GHY-functional
\begin{equation}\label{secondghy}
\def\arraystretch{2.2}
\begin{array}{r>{\displaystyle}l}

\delta \ghy_{g}(h) = & -\int_M\left(R_{ij}-\frac{1}{2}g_{ij}\scal{g}\right)h^{ij} \fv{g}\\
 & - \int_{\partial M} \left(\II_{\alpha\beta}-\mean{g} g_{\alpha\beta}\right)h^{\alpha\beta}\fvr{g}.\\
\end{array}
\end{equation}

\smallskip

Now, we will focus on the computation of the second variation. For this case, and for simplicity, only conformal variations will be considered (for the general case, see \cite{A}). So, let us assume now that $h=fg$ for a positive function $f\in\mathcal{H}^1(M) $. With this assumption, observe that (\ref{cabfirstdev}) and (\ref{ghyfirst1}) become
\begin{align}\label{firstcabfg}
\delta\,\Cab(fg)= \frac{an}{2}\int_{M}f\fv{g} + \frac{b(n-1)}{2}\int_{\bord{2}}f\fvr{g_2},
\end{align}
and
\begin{align}\label{firstFfg}
\delta\, \ghy(fg) = & \frac{n-2}{2}\int_{M} \scal{g}\,f\fv{g}+ (n-2)\int_{\partial M} \mean{g}\,f \fvr{g},
\end{align}
respectively. Therefore, the second variation of the former is given by
\begin{equation}\label{cabsecond}
\delta^2\, (\Cab)_{g}(fg)= \frac{n^2\,a}{4}\int_{M} f^2 \fv{g} + \frac{(n-1)^2\,b}{4}\int_{\bord{2}} f^2 \fvr{g_2},
\end{equation}
where we have used that
\[
\delta\,\fv{fg}=\frac{n}{2}f \fv{g},\qquad\delta\fvr{fg}=\frac{n-1}{2}f\fvr{g_2}.
\]
For the second variation of the latter, observe that both (\ref{scalarfirst}) and (\ref{firstmean}) become
\[
\delta \scal{fg}=-f\scal{g} + (n-1)\Delta_g{f}, \qquad \delta \mean{fg}=-\frac{1}{2}f\mean{g}- \frac{n-1}{2}\partial_{n}f,
\]
where $\Delta_g=-\mathrm{div}_g(\grad)$ is the Laplace-Beltrami operator.

Finally, the second variation of the functional $F$ for conformal variations has the following form
\begin{equation}\label{secondvariationghy}
\def\arraystretch{2.2}
\begin{array}{r>{\displaystyle}l}
\delta^2 (\ghy)_{g} (fg)= & \frac{n-2}{2}\int_{M}f\left(\scal{g}\left(\frac{n-2}{2}\right)f + (n-1)\Delta_g{f}\right)\fv{g}\\ & + (n-2)\int_{\partial M}f\left( \left(\frac{n-2}{2}\right) \mean{g}f-\left(\frac{n-1}{2}\right)\partial_{n}f\right)\fvr{g}.
\end{array}
\end{equation}
%
%
%
%
\section{Acknowledgements}
The first author was partially supported by FAPESP (Fundac\~ao de Amparo \`{a} Pesquisa do Estado de S\~ao Paulo, Brazil) Process 2012/22490-7 and by MINECO-FEDER project MTM2012-34037 and Fundaci\'{o}n S\'{e}neca
project 04540/GERM/06, Spain. This research is a result of the activity developed within the framework of the Programme in Support of Excellence Groups of the Regi\'{o}n de Murcia, Spain, by Fundaci\'{o}n S\'{e}neca, Regional Agency for Science and Technology (Regional Plan for Science and Technology 2007-2010).

The second author was supported by FAPESP (Funda\c{c}\~ao de Amparo \'{a} Pesquisa do Estado de S\~ao Paulo, Brazil) Process 2012/11950-7 and the Spanish Grants MTM2013-47828-C2-1-P and MTM2013-47828-C2-2-P (MINECO AND FEDER funds).
\bibliographystyle{amsplain}

\end{document}